\documentclass[11pt, oneside, english]{article}

\usepackage[english]{babel}
\usepackage[utf8x]{inputenc}
\usepackage[T1]{fontenc}
\usepackage[letterpaper,top=1in,bottom=1in,left=1in,right=1in,marginparwidth=1.75cm]{geometry}
\usepackage{dsfont}
\usepackage{amsmath, amssymb, amsthm, verbatim,enumerate,bbm}
\usepackage{indentfirst, bbm}
\usepackage{appendix}
\usepackage{enumerate}
\usepackage{verbatim}
\usepackage[colorinlistoftodos]{todonotes}
\usepackage[colorlinks=true, allcolors=blue]{hyperref}
\usepackage[nameinlink,capitalise,noabbrev]{cleveref}
\crefname{appsec}{Appendix}{Appendices}
\creflabelformat{enumi}{#2textup{#1}#3}

\allowdisplaybreaks
\makeatletter
\newtheorem{theorem}{Theorem}[section]

\newtheorem*{namedtheorem}{\theoremname}
\newcommand{\theoremname}{testing}

\newtheorem{thm}[theorem]{Theorem}
\newtheorem{lemma}[theorem]{Lemma}

\newtheorem{proposition}[theorem]{Proposition}

\newtheorem*{question*}{Question}

\theoremstyle{definition}

\newtheorem{remark}[theorem]{Remark}

\theoremstyle{plain}

\makeatother

\title{Universality and least singular values of random matrix products: a simplified approach}
\author{
    Rohit Chaudhuri\thanks{Email: {\tt  rohitchaudhuri@g.harvard.edu}.}
     \and
    Vishesh Jain \thanks {Email: {\tt  vishesh.vj@gmail.com}.}
      \and
   Natesh S. Pillai \thanks{Email: {\tt  pillai@fas.harvard.edu}.}
}
\date{}

\DeclareMathOperator{\Sparse}{Sparse}
\DeclareMathOperator{\Comp}{Comp}
\DeclareMathOperator{\Incomp}{Incomp}
\DeclareMathOperator{\op}{op}
\DeclareMathOperator{\dist}{dist}

\begin{document}
\maketitle
\global\long\def\R{\mathbb{R}}

\global\long\def\S{\mathbb{S}}

\global\long\def\Z{\mathbb{Z}}

\global\long\def\C{\mathbb{C}}

\global\long\def\Q{\mathbb{Q}}

\global\long\def\N{\mathbb{N}}

\global\long\def\P{\mathbb{P}}

\global\long\def\F{\mathbb{F}}

\global\long\def\U{\mathcal{U}}

\global\long\def\V{\mathcal{V}}

\global\long\def\E{\mathbb{E}}

\global\long\def\Ev{\mathscr{Rk}}

\global\long\def\Dg{\mathscr{D}}

\global\long\def\Ndg{\mathscr{ND}}

\global\long\def\Rv{\mathcal{R}}

\global\long\def\Gv{\mathscr{Null}}

\global\long\def\Hv{\mathscr{Orth}}

\global\long\def\Supp{{\bf Supp}}

\global\long\def\Sv{\mathscr{Spt}}

\global\long\def\ring{\mathfrak{R}}

\global\long\def\1{\mathbbm{1}}

\global\long\def\Bad{{\boldsymbol{B}}}

\global\long\def\supp{{\bf supp}}

\global\long\def\A{\mathcal{A}}

\global\long\def\L{\mathcal{L}}

\begin{abstract}
In this note, we show how to provide sharp control on the least singular value of a certain translated linearization matrix arising in the study of the local universality of products of independent random matrices. This problem was first considered in a recent work of Koppel, O'Rourke, and Vu, and compared to their work, our proof is substantially simpler and established in much greater generality . In particular, we only assume that the entries of the ensemble are centered, and have second and fourth moments uniformly bounded away from $0$ and infinity, whereas previous work assumed a uniform subgaussian decay condition and that the entries within each factor of the product are identically distributed.  

A consequence of our least singular value bound is that the four moment matching universality results for the products of independent random matrices, recently obtained by Koppel, O'Rourke, and Vu, hold under much weaker hypotheses. Our proof technique is also of independent interest in the study of structured sparse matrices. 
\end{abstract}

\section{Introduction}

Let $\boldsymbol{X}^{(1)},\boldsymbol{X}^{(2)},\dots,\boldsymbol{X}^{(M)}$ be mutually independent $n \times n$ (complex) random matrices with independent entries. In a recent work, Koppel, O'Rourke and Vu \cite{kopel2018random} studied the local universality of correlation functions associated with the product $\boldsymbol{X}:= \boldsymbol{X}^{(1)}\circ \dots \circ \boldsymbol{X}^{(M)}$, as well as the limits of linear spectral statistics of $\boldsymbol{X}$, under a four moment matching hypothesis, in the limit $n\to \infty$ while holding $M$ fixed; we refer the reader to \cite{kopel2018random} for an introduction to this area, as well as an extensive bibliography. The key technical contribution of their work was a lower bound on the smallest singular value of a certain translated linearization matrix associated with $\boldsymbol{X}$, which was established using a careful consideration of the interplay between the block structure of this translated linearization matrix, as well as the linear spaces spanned by the smallest singular vectors of the individual factor matrices $\boldsymbol{X}^{(i)}$. 

The goal of this note is to show how to establish such a bound on the smallest singular value in a very simple manner under considerably weaker assumptions than in \cite{kopel2018random} and only using (by now) completely standard arguments. Notably, unlike in \cite{kopel2018random}, the smallest singular values/vectors of the individual factor matrices play no role in our analysis; in fact, the only information we need about the factor matrices is a good bound on the \emph{largest} singular value, as well as a uniform anti-concentration assumption on the individual entries. Combining our singular value result with the proofs in \cite{kopel2018random} in a black-box manner shows that essentially all of the main results in \cite{kopel2018random} can be established in much greater generality. We also believe that our proof techniques should be more generally useful in the study of structured sparse matrices. These applications are discussed in more detail in \cref{subsec:applications}. 

\subsection{The smallest singular value of the translated linearized matrix}
Let $\boldsymbol{X}^{(1)},\dots, \boldsymbol{X}^{(M)}$ be $n\times n$ complex matrices, and consider the associated $Mn \times Mn$ \emph{linearization} block matrix $\boldsymbol{Y}$ given by
\[
\boldsymbol{Y}=
  \begin{bmatrix}
    \boldsymbol{0} & \boldsymbol{X}^{(1)} & \boldsymbol{0} & \dots & \dots & \boldsymbol{0}\\
     \boldsymbol{0} & \boldsymbol{0} & \boldsymbol{X}^{(2)}  & \dots & \dots & \boldsymbol{0}\\
     \boldsymbol{0} & \boldsymbol{0}  & \boldsymbol{0} & \boldsymbol{X}^{(3)} & \dots & \boldsymbol{0}\\
     \vdots & \vdots & \vdots & \vdots & \ddots & \vdots \\
     \boldsymbol{0} & \boldsymbol{0} & \boldsymbol{0} & \boldsymbol{0} & \dots & \boldsymbol{X}^{(M-1)}\\
     \boldsymbol{X}^{(M)} & \boldsymbol{0} & \boldsymbol{0} & \boldsymbol{0} & \dots & \boldsymbol{0}
  \end{bmatrix}
\]
The utility of $\boldsymbol{Y}$ in studying the product $\boldsymbol{X} = \boldsymbol{X}^{(1)}\circ \dots \circ \boldsymbol{X}^{(M)}$ goes back to the work of Burda, Janik, and Waclaw \cite{burda2010spectrum}, who observed that if $\lambda_1,\dots, \lambda_{n}$ are the eigenvalues of $\boldsymbol{X}$, then each $\lambda_{k}$ is an eigenvalue of $\boldsymbol{Y}^M$ with multiplicity $M$. 

For us, what will be useful is not $\boldsymbol{Y}$, but rather, a translation of it by a complex number of modulus approximately $\sqrt{n}$. More precisely, for $z \in \C$, we define the $Mn\times Mn$ matrix $\boldsymbol{Y}(z)$ by $\boldsymbol{Y} - z\boldsymbol{I}$, i.e. 
\[
\boldsymbol{Y}(z)=
  \begin{bmatrix}
    -z\boldsymbol{I} & \boldsymbol{X}^{(1)} & \boldsymbol{0} & \dots & \dots & \boldsymbol{0}\\
     \boldsymbol{0} & -z\boldsymbol{I} & \boldsymbol{X}^{(2)}  & \dots & \dots & \boldsymbol{0}\\
     \boldsymbol{0} & \boldsymbol{0}  & -z\boldsymbol{I} & \boldsymbol{X}^{(3)} & \dots & \boldsymbol{0}\\
     \vdots & \vdots & \vdots & \vdots & \ddots & \vdots \\
     \boldsymbol{0} & \boldsymbol{0} & \boldsymbol{0} & \boldsymbol{0} & \dots & \boldsymbol{X}^{(M-1)}\\
    \boldsymbol{X}^{(M)} & \boldsymbol{0} & \boldsymbol{0} & \boldsymbol{0} & \dots & -z\boldsymbol{I}
  \end{bmatrix}
\]
As noted in \cite{kopel2018random} (see the proof of Corollary 8 there), the top left $n\times n$ minor of $\boldsymbol{Y}(z)^{-1}$ is given by 
$$\left(\frac{1}{z^{M-1}}\boldsymbol{X} - z\boldsymbol{I} \right)^{-1},$$
so that a lower bound on the smallest singular value of $\boldsymbol{Y}(z)$ automatically gives a lower bound on the smallest singular value of $\boldsymbol{X} - z^{M}\boldsymbol{I}$.\\

Our main result is the following generalization of Theorem 7 in \cite{kopel2018random}. 
 \begin{thm}
 \label{thm:main}
 Fix $M \geq 1$, and for $k \in [M]$, let $\boldsymbol{X}^{(k)} = (\xi^{(k)}_{i,j})_{i,j}$ be $n\times n$ complex random matrices such that the random variables $\{\xi^{(k)}_{i,j}\}_{i,j,k}$ are mutually independent. Suppose that for all $i,j,k$,
 \begin{itemize}
 \item $\E[\xi_{i,j}^{(k)}] = 0$,
 \item $\E[|\xi^{(k)}_{i,j}|^{2}]\geq c_{2}$, and
 \item $\E[|\xi^{(k)}_{i,j}|^{4}] \leq C_4$,
 \end{itemize}
 where $c_2, C_4$ are positive constants.
For $z\in \C$, let $\boldsymbol{Y}(z)$ be the $Mn\times Mn$ matrix defined above, and let $\sigma_{1}(\boldsymbol{Y}(z))$ denote its smallest singular value. Then, there exist constants $A_{\ref{thm:main}}$ and $C_{\ref{thm:main}}$ depending only on $c_2, C_4, M$ such that for all fixed $A \in (0, A_{\ref{thm:main}})$, if $n^{1/2 - A/1000M}\leq |z| \leq n^{1/2 + A/1000M}$, then
$$\P\left(\sigma_{1}(\boldsymbol{Y}(z)) \leq n^{-1/2 - A}\right) \leq C_{\ref{thm:main}}n^{-A/1000M}.$$
\end{thm}
\begin{remark}
In \cite{kopel2018random}, a version of this theorem is proved under the following two additional assumptions: for each fixed $k \in [M]$, the random variables $\{\xi^{(k)}\}_{i,j}$ are assumed to be i.i.d. (we only require independence), and the collection of centered random variables $\{\xi^{(k)}\}_{i,j}$ are assumed to be uniformly subgaussian (we only require a uniform upper bound on the fourth moment and a uniform lower bound on the second moment). We note that especially for the application to the universality of random matrix products, our moment assumptions are not restrictive, since in such applications, we have a moment matching assumption on the first four moments of the distributions $\{\xi^{(k)}\}_{i,j}$ anyway.
\end{remark}

By using the general strategy of Rudelson and Vershynin \cite{rudelson2008littlewood}, controlling the least singular value of $\boldsymbol{Y}(z)$ boils down to bounding from below the distance between the last row (say) of $\boldsymbol{Y}(z)$ and the span of the first $Mn-1$ rows. The key challenge here is to show that, with high probability, any unit vector orthogonal to the first $Mn-1$ rows of $\boldsymbol{Y}(z)$ has a nearly constant fraction of its $\ell_{2}$-mass concentrated in its first $n$ entries (note that there is \emph{some} block of $n$ entries containing at least $1/\sqrt{M}$ fraction of the $\ell_2$-mass) -- this is a precursor to being able to use any anti-concentration estimates to lower bound this distance. In contrast to \cite{kopel2018random}, where this challenge is overcome by considering the delicate interaction between the linear spaces spanned by the small singular vectors of the factor matrices, our proof of this key step (see part (1) of \cref{thm:key}) only makes use of the operator norm of the factor matrices, a trivial anti-concentration property of the entries of the factor matrices, and the cyclicness of the block structure of $\boldsymbol{Y}(z)$. 

In slightly more detail, let $u = (u^{(1)},\dots, u^{(M)}) \in \C^{Mn}$ denote a unit vector orthogonal to the first $Mn-1$ rows of $\boldsymbol{Y}(z)$. Then, by using the fact that the operator norm of each factor matrix is (with high probability) almost the same as the magnitude of $z$, and using the first $(M-1)n$ equations satisfied by $u$, we can easily show that $\|u^{(M)}\|_{2} = \Omega(n^{-\epsilon})$. To transfer this conclusion to $\|u^{(1)}\|_{2}$, we consider two cases: 
\begin{itemize}
\item If the first $n-1$ coordinates of $u^{(M)}$ contain a non-vanishing fraction of its $\ell_{2}$-mass, then using the same fact that the operator norm of each factor matrix is almost the same as the magnitude of $z$, along with the last $n-1$ equations satisfied by $u$, shows that $\|u^{(1)}\|_{2} = \Omega(n^{-O(\epsilon)})$. 
\item If the first case does not hold, then almost all of the $\ell_2$-mass of $u^{(M)}$ is concentrated on the last coordinate. But then, the first $(M-1)n$ equations satisfied by $u$, along with the fact that the image of any \emph{fixed} vector under any of the factor matrices has norm $\Omega(\sqrt{n})$ (which is almost the same as the magnitude of $z$) gives the desired conclusion. 
\end{itemize}
\subsection{Applications of \cref{thm:main}}
\label{subsec:applications}
As remarked earlier, \cref{thm:main} can be used to establish the main results in \cite{kopel2018random} under much more general conditions. As an illustration, one can generalize Theorem 3 in \cite{kopel2018random} as follows:
\begin{theorem}
\label{thm:linear-stat}
  Let $f\colon \C \to \R$ be a function with at least two continuous derivatives, supported in the \emph{spectral bulk} $\{z\in \C : \tau_0 < |z| < 1-\tau_0\}$ for some fixed $\tau_0 > 0$. Fix an integer $M \geq 1$, and let $n^{-M/2}\boldsymbol{X}^{(1)}\circ \dots \circ \boldsymbol{X}^{(M)}$ (with eigenvalues $\lambda_1,\dots,\lambda_n$) be a matrix product such that each factor is an $n\times n$ independent random matrix whose entries match the standard complex Gaussian distribution to four moments. Finally, suppose that the random variables $\{\xi_{i,j}^{(k)}\}_{i,j,k}$ are uniformly subexponential. Then, as $n\to \infty$, the centered linear statistic
 $$N_{n}[f] = \sum_{j}f(\lambda_j) - \sum_{j}\E[f(\lambda_j)] $$
 converges in distribution to the centered normal distribution with variance
 $$ \frac{1}{4\pi}\int_{|z| < 1}|\nabla f(z)|^{2} d^{2}z + \frac{1}{2}\|f\|^{2}_{H^{1/2}(|z|=1)}.$$
\end{theorem}
 Here, $\|\cdot \|_{H^{1/2}(|z|=1)}$ is the $1/2$-Sobolev norm of the function restricted to the unit disc (see Definition 2 in \cite{kopel2018random}).
\begin{remark}
In \cite{kopel2018random}, this theorem is proved under the additional assumptions that the entries of each factor matrix are i.i.d. and that the collection of all of the $Mn^{2}$ random variables is uniformly subgaussian. We note that the uniform subexponential assumption in \cref{thm:linear-stat} stems from the work of Nemish on the local $M$-fold circular law \cite{nemish2017local}; any weakening of the assumptions of this work directly lead to a corresponding improvement in \cref{thm:linear-stat}. Finally, we note that similar improvements as in \cref{thm:linear-stat} also hold for Theorems 5 and 6 in \cite{kopel2018random}.
\end{remark}

We also believe that the proof of our key technical proposition (\cref{thm:key}) is of interest in the study of independent random matrices with structured sparsity i.e. random matrices with independent entries such that a prescribed collection of entries are equal to $0$ almost surely. Examples of structured sparse matrices include (non-Hermitian) band and block-band matrices (see e.g. the  many references in \cite{jjlo2020}). As is the case here, the key technical challenge in bounding the smallest singular value of random matrices with structured sparsity is to show that any unit vector which is orthogonal to all but one row of the matrix has `substantial overlap' with the support of the remaining row. Our proof of \cref{thm:key} shows how, in certain cases, one may leverage the structure of the sparsity to obtain such a conclusion in a straightforward manner. Indeed, in upcoming work of Jana, Luh, O'Rourke and the second named author \cite{jjlo2020}, a similar idea as \cref{thm:key} is used to obtain a crucial least singular value estimate used to prove, for the first time, a circular law for random block band matrices with bandwidth $n^{1-\delta}$ (for an absolute constant $\delta > 0$).

\section{Proof of \cref{thm:main}}

\subsection{Estimates on the operator norm}
For any matrix $\boldsymbol{A}$, let $\|\boldsymbol{A}\|_{\op} $ denote its operator norm. We begin with the following bound on the expectation of the operator norm due to Lata{\l}a \cite{latala2005some}. 
 \begin{thm}[\cite{latala2005some}]\label{thm:latala}
For any $n\times n$ complex random matrix $\boldsymbol{X}:=(X_{ij})_{1\leq i,j \leq n}$ with independent centered entries, $$\mathbb{E}\|\boldsymbol{X}\|_{\op}\leq C\left(\max_{i}\sqrt{\sum_{j}\mathbb{E}|X_{ij}|^{2}}+\max_{j}\sqrt{\sum_{i}\mathbb{E}|X_{ij}|^{2}}+\sqrt[\leftroot{0}\uproot{3}4]{\sum_{i,j}\mathbb{E}|X_{ij}|^{4}}\right),$$
where $C$ is some universal constant.
\end{thm}

In particular, we see that for any  $n \times n$ complex random matrix $\boldsymbol{X}$ with independent centered entries $X_{i,j}$ satisfying $\E |X_{i,j}|^{4} \leq C_{4}$ for all $1\leq i,j \leq n$,   $\mathbb{E}\|\boldsymbol{X}\|_{\text{op}}=O_{C_4}(\sqrt{n})$. As an immediate corollary, we have the following. 
\begin{lemma}\label{thm:norm}
For each $i\in [M]$, let $\boldsymbol{X}^{(i)}$ be an $n\times n$ complex random matrix with independent centered entries with fourth moments bounded by $C_{4}$. Then, for a given $\epsilon_{0}>0$, except with probability at most $O_{C_{4}}(n^{-\frac{\epsilon_0}{32M}})$, 
$$\max \left\{\|\boldsymbol{X}^{(1)}\|_{\op},\dots,\|\boldsymbol{X}^{(M)}\|_{\op}\right\} = O_{C_4,M}\left(n^{\frac{1}{2}+\frac{\epsilon_{0}}{32M}}\right).$$ 
\end{lemma}

\begin{proof}
By Markov's inequality and \cref{thm:latala}, it follows that for each $i \in [M]$, 
$$\P\left(\|\boldsymbol{X}^{(i)}\|_{\op} \geq O_{C_4}\left(M\cdot n^{\frac{1}{2}+ \frac{\epsilon_0}{32M}}\right)\right) = O_{C_{4}}(M^{-1}\cdot n^{-\frac{\epsilon_0}{32M}}).$$
Taking the union bound over $i \in [M]$ gives the desired conclusion. \qedhere

\end{proof}    

\begin{remark}
The previous lemma, along with the triangle inequality, implies that for $z \in \mathbb{C}$ with $|z|\leq n^{1/2+\epsilon_0/32M} $, we have  
$\|\textbf{Y}(z)\| \leq O(n^{\frac{1}{2}+\frac{\epsilon_{0}}{32M}})$. 
\end{remark}
%

\subsection{Compressible and incompressible vectors}
Throughout the remainder of this paper, we set $N = Mn$. For any integer $d \geq 1$, 
let $\S_{\C}^{d-1}:=\{(x_1,\dots,x_d) \in \C^{d} : |x_1|^{2} + \dots +|x_d|^{2} = 1\}$ denote the set of unit vectors in $\C^{d}$ (equipped with the Euclidean norm).
We will need the following standard decomposition of $\S_{\C}^{d-1}$ due to Rudelson and Vershynin \cite{rudelson2008littlewood}. For parameters $a,b \in (0,1)$, let $\Sparse_{d}(a)$ denote the set of vectors in $\C^{d}$ with support of size at most $ad$, let $\Comp_{d}(a,b)$ denote the set of vectors in $\S_{C}^{d-1}$ which have Euclidean distance at most $b$ from $\Sparse_{d}(a)$, and let $\Incomp_{d}(a,b)$ denote the set of vectors in $\S_\C^{d-1}$ which are not compressible. From here on, we will drop the subscript $d$ when the underlying dimension is clear. Then, by the union bound, for any $a,b \in (0,1)$,    
\begin{align*}
\mathbb{P}\left\{\sigma_{1}(\boldsymbol{Y}(z))\leq n^{-1/2-A}\right\} 
&\leq \\ 
\mathbb{P}\left\{\inf_{v \in \Comp(a,b)}\|\boldsymbol{Y}(z)v\|_{2}\leq n^{-1/2-A}\right\} &+\mathbb{P}\left\{\inf_{u \in \Incomp(a,b)}\|\boldsymbol{Y}(z)u\|_{2}\leq  n^{-1/2-A}\right\}.
\end{align*}
We will bound each of these two terms using a separate argument. The bound on the first term is easier, and is discussed in the next subsection, following which we present the bound on the second term. 
\subsection{Compressible Case }
The following lemma shows that, with high probability, the image of a fixed vector under $\boldsymbol{Y}(z)$ is far away from any fixed vector. More precisely, we have: 
\begin{lemma}
\label{thm:anti}
 For any fixed $z \in \mathbb{C}$, $u \in \S_{\C}^{N-1}$, and $w \in \C^{N}$, $$\mathbb{P}\left\{\|\boldsymbol{Y}(z) u - w\|_{2} \leq c_{\ref{thm:anti}} \sqrt{n}\right\} =  O_{C_4,c_2}(\exp{(-c_{\ref{thm:anti}}n)}).$$ 
 Here, $c_{\ref{thm:anti}}$ is a constant depending only on $C_4, c_2$ and $M$. 
\end{lemma}
\begin{proof}
Writing $u = (u^{(1)},\dots, u^{(M)})$, we may assume without loss of generality that $\|u^{(1)}\|_{2} \geq M^{-1/2}$. Moreover, since $\|\boldsymbol{Y}(z)u - w\|_{2} \geq \|\boldsymbol{X}^{(M)}u^{(1)}- z u^{(M)} - w^{(M)}\|_{2}$, it suffices to bound from above the probability that the latter quantity is small. 

For this, we begin by noting that since for all $i,j \in [n]$, $\xi^{(n)}_{i,j}$ is a random variable with second moment at least $c_2$ and fourth moment at most $C_4$, the Paley-Zygmund inequality implies that there exists some constant $c$ depending only on $c_2, C_4$ such that for all $i,j \in [n]$, $\Pr\left(|\tilde{\xi}^{(n)}_{i,j}| \geq c\right) \geq c$, where $\tilde{\xi}^{(n)}_{i,j}$ denotes the symmetrization of ${\xi}^{(n)}_{i,j}$. Using again the bound on the fourth moment, it follows from Markov's inequality that, in fact, $c$ may be taken to be sufficiently small (again, depending only on $c_2$ and $C_4$) so that
$$\Pr\left(c^{-1} \geq |\tilde{\xi}^{(n)}_{i,j}| \geq c\right) \geq c.$$
The desired conclusion now follows directly from Lemma 2.8 in \cite{jain2020note} along with the standard tensorization lemma (Lemma 2.2 in \cite{rudelson2008littlewood}). \end{proof}

Given this lemma, the compressible case is handled using a (by now) standard argument, which we reproduce here for the reader's convenience. At the crux of the argument is the low metric entropy of unit vectors in $\Comp_{N}(a,b)$. Indeed, as \cite{rudelson2008littlewood} shows:

\begin{lemma}\label{thm:net}
For $a, b \in (0,1/8)$,  $\Comp_{N}(a,b)$ admits a $2b$-net with cardinality at most:
$$\left(\frac{C_{\ref{thm:net}}}{ab}\right)^{2aN},$$
where $C_{\ref{thm:net}}>0$ is an absolute constant.
\end{lemma}

\begin{lemma}
\label{lemma:comp}
For every sufficiently small (depending only on $c_2, C_4, M$) $\epsilon_{0} > 0$, 
the following estimate holds for all $|z|\leq n^{1/2+\epsilon_0/32M}$, $a=1/\log(n)$ and for $b= n^{-\epsilon}$ (for all $\epsilon_0/16M < \epsilon < 100\epsilon_0$): $$\mathbb{P}\left\{\inf_{u \in \Comp(a,b)}\|\boldsymbol{Y}(z)u\|_{2}\leq c_{\ref{lemma:comp}}\sqrt{n}\right\}= O_{c_2, C_4, M}(n^{-\frac{\epsilon_{0}}{32M}}),$$
where $c_{\ref{lemma:comp}}$ is a constant depending only on $c_2, C_4, M$. 
\end{lemma}
\begin{proof}

Let
$$\mathcal{A}:=\left\{\|\boldsymbol{Y}(z)\|_{\op}= O_{C_4,M}\left(n^{\frac{1}{2} + \frac{\epsilon_0}{32M}}\right)\right\}.$$ 
Then, by \cref{thm:norm} and the triangle inequality,  $$\mathbb{P}(\mathcal{A})\geq 1-O_{C_4}(n^{-\frac{\epsilon_{0}}{32M}}).$$   
Therefore, we may restrict ourselves to the event $\mathcal{A}$.

For $(a,b) \in (0,1/8)$, let $\Omega_{a,2b}$ denote a $2b$-net of $\Comp_{N}(a,b)$. Then, by \cref{thm:net}, \cref{thm:anti}, and the union bound,   
$$ \P \left\{\inf_{u \in \Omega_{a,2b}}\|\boldsymbol{Y}(z)u\|_{2} \leq c_{\ref{thm:anti}}\sqrt{n} \right\} = O_{C_4, c_2}\left(\exp(-c_{\ref{thm:anti}}n)\exp\left(2aN\log\left(\frac{C_{\ref{thm:net}}}{ab}\right)\right)\right).$$
In particular, by taking $a = 1/(\log{n})$ and $b = n^{-\epsilon}$ (for all $\epsilon < \epsilon_1$, where $\epsilon_1 > 0$ depends only on $c_2, C_4, M$), it follows that
$$\P \left\{\inf_{u \in \Omega_{a,2b}}\|\boldsymbol{Y}(z)u\|_{2} \leq c_{\ref{thm:anti}}\sqrt{n} \right\} = O_{C_4, c_2}\left(\exp(-cn)\right),$$
where $c > 0$ depends only on $c_2, C_4, M$. 


By definition, for any $u \in \Comp(a,b)$, there exists some $u' \in \Omega_{a,2b}$ such that $\|u - u'\|_{2} \leq 2b$. Therefore, since we have restricted ourselves to the event $\mathcal{A}$, we have for all $n$ sufficiently large (depending on $c_2, C_4, M$) that 
\begin{align*}
    \|\boldsymbol{Y}(z) u\|_{2}
    &\geq \|\boldsymbol{Y}(z)u'\|_{2} - \|\boldsymbol{Y}(z)(u-u')\|_{2}\\
    &\geq c_{\ref{thm:anti}}\sqrt{n} - \|\boldsymbol{Y}(z)\|_{\op}\|u-u'\|_{2}\\
    &\geq c_{\ref{thm:anti}}\sqrt{n} - 2\|\boldsymbol{Y}(z)\|_{\op}b\\
    &\geq c_{\ref{thm:anti}}\sqrt{n}/2,
\end{align*}
provided that $b < n^{-\epsilon_0/16M}$,
which completes the proof. 
\end{proof}
\subsection{Incompressible Case}
The goal of this section is to prove the following.

\begin{lemma}
\label{lemma:incomp}
For every sufficiently small (depending only on $c_2, C_4, M$) $\epsilon_{0} \in (0,100)$, 
the following estimate holds for all $|z|\leq n^{1/2+\epsilon_0/32M}$, $a=1/\log(n)$ and for $b= n^{-\epsilon}$ (for all $\epsilon_0/4 < \epsilon < 100\epsilon_0$): $$\mathbb{P}\left\{\inf_{u \in \Incomp(a,b)}\|\boldsymbol{Y}(z)u\|_{2} \leq c_{\ref{lemma:incomp}}n^{-\epsilon_0}b^{5}(Mn)^{-1/2}\right\}= O_{c_2, C_4, M}(b^{2} + n^{-\epsilon_0/32M}),$$
where $c_{\ref{lemma:incomp}}$ is a positive constant depending only on $c_2, C_4, M$.
\end{lemma}

The first step in proving \cref{lemma:incomp} is the following `invertibility-via-distance' lemma due to Rudelson and Vershynin \cite{rudelson2008littlewood}. 

\begin{lemma}\label{lemma:distance}
Let $\dist_{k}$ denote the distance between the $k$-th row of $\boldsymbol{Y}(z)$ and the span of the remaining $Mn-1$ rows. Then, for any $(a,b) \in (0,1/4)$, and for any $\varepsilon' \geq 0$, 
$$\mathbb{P}\left\{\inf_{v \in \Incomp(a,b)}\|\boldsymbol{Y}(z)v\|_{2}\leq \varepsilon'b(Mn)^{-1/2}\right\}\leq \frac{1}{aMn}\sum_{k=1}^{Mn}\mathbb{P}\left\{\dist_{k}\leq \varepsilon'\right\}.$$ \end{lemma}

In the majority of the remainder of this subsection, we will estimate $\P\left\{\dist_{Mn} \leq \epsilon'\right\}$. The argument for other values of the index $k$ follows by purely notational changes. The key estimate we need to bound $\P\left\{\dist_{Mn} \leq \epsilon'\right\}$ is the following analog of Lemmas 22 and 23 in \cite{kopel2018random}. While our proof of the part corresponding to Lemma 23 broadly follows the proof in  \cite{kopel2018random}, we provide a completely elementary proof of the substantially more challenging Lemma 22 in \cite{kopel2018random}. Indeed, the elementary proof of this lemma is the source of most of our simplifications and allows us to remove many of the technical hypotheses in \cite{kopel2018random}. 

In the following, for a $k\times \ell$ matrix $\boldsymbol{A}$, we will use $\Tilde{\boldsymbol{A}}$ to denote the $(k-1) \times \ell$ matrix resulting from removing the last row of $\boldsymbol{A}$. 
\begin{proposition}\label{thm:key}
For every sufficiently small (depending only on $c_2, C_4, M$) $\epsilon_0 > 0$, and $t_0 = \epsilon_0/32M$, the following holds.  
Suppose that $n^{1/2-\epsilon_{0}/32M}\leq |z| \leq n^{1/2+\epsilon_{0}/32M}$, and let $u \in \S_{\C}^{N-1}$ be any unit vector orthogonal to the subspace spanned by the first $Mn-1$ rows of $\boldsymbol{Y}(z)$. Writing $u=(u^{(1)},\dots,u^{(M)}),$ where each $u^{(i)} \in \C^{n}$, we have,
\begin{enumerate}
\item Over the choice of $\boldsymbol{X}^{(1)}\dots \boldsymbol{X}^{(M-1)}, \Tilde{\boldsymbol{X}}^{(M)}$,
$$\mathbb{P}\left\{\min\{\|u^{(1)}\|_{2}, \|u^{(M)}\|_{2}\} = \Omega_{C_4,c_2,M}\left( n^{-\epsilon_{0}}\right)\right\} = 1 - O_{C_4,M}\left(n^{-\frac{\epsilon_{0}}{32M}}\right);$$
\item For $a = 1/\log{n}$ and $b = n^{-\epsilon}$ (for any $8Mt_0 < \epsilon < 100\epsilon_0$),
we have with probability at least $1 - O_{C_4,M}\left(n^{-\frac{\epsilon_{0}}{32M}}\right)$ that $u^{(1)}/\|u^{(1)}\|_{2}$ is a well defined element of $\Incomp_{n}(a,b)$. 
\end{enumerate}
\end{proposition}
\begin{proof}
We begin by proving the first conclusion of the proposition. Let $\mathcal{A}_{1}$ and $\mathcal{A}_{2}$ be the events given by
\begin{align*}
\mathcal{A}_{1}&:=\left\{\max\left\{\|\boldsymbol{X}^{(1)}\|_{\op},\dots,\|\boldsymbol{X}^{(M-1)}\|_{\op}, \|\Tilde{\boldsymbol{X}}^{(M)}\|_{\op} \right\} \leq O_{C_4,M}\left(n^{\frac{1}{2}+\frac{\epsilon_{0}}{32M}}\right)\right\}\text{ and } \\
\mathcal{A}_{2}&:=\left\{\left\|\left(\frac{\boldsymbol{X}^{(1)}}{z}\right)\left(\frac{\boldsymbol{X}^{(2)}}{z}\right)\dots\left(\frac{\boldsymbol{X}^{(M-1)}}{z}\right){e}_{n}\right\|_{2}\geq \left(\frac{c_{\ref{thm:anti}}\sqrt{n}}{|z|}\right)^{M-1}\right\},
\end{align*}
where $e_{n}$ denotes the vector with $1$ in the $n^{th}$ coordinate and $0$ everywhere else. A simple iterative argument, using the mutual independence of the matrices $\boldsymbol{X}^{(i)}$, \cref{thm:anti} and the sequence of vectors (for $j = M-1,M-2,\dots,1$) given by
$$y^{(j)} := \left(\prod_{k=j}^{M-1} \boldsymbol{X}^{(k)}\right)e_{n}$$
shows that
$$ \P(\mathcal{A}_{2}) \geq 1 - O_{C_4, c_2, M}(\exp(-c_{\ref{thm:anti}}n)).$$

Moreover, from \cref{thm:norm}, $\mathcal{A}_1$ occurs with with probability at least $1-O_{C_4,M}\left(n^{-\frac{\epsilon_{0}}{32M}}\right)$. Hence, up to losing an additive factor of $O_{C_4,c_2,M}\left(n^{-\frac{\epsilon_0}{32M}}\right)$, we may restrict ourselves to the event 
\begin{align}
\label{eqn:A}
\mathcal{A} = \mathcal{A}_1 \cap \mathcal{A}_2.
\end{align}

Since, $u$ is normal to the span of the first $nM-1$ rows of $\boldsymbol{Y}(z)$ by assumption, we must have that
\begin{align}
\label{eqn:system}
zu^{(1)} &= \boldsymbol{X}^{(1)}u^{(2)} \nonumber\\
zu^{(2)} &= \boldsymbol{X}^{(2)}u^{(3)} \nonumber\\
\vdots \\
zu^{(M-1)} &= \boldsymbol{X}^{(M-1)}u^{(M)} \nonumber\\
z\Tilde{\boldsymbol{I}}u^{(M)}&=\Tilde{\boldsymbol{X}}^{(M)}u^{(1)}. \nonumber
\end{align}
In particular, we have for $1\leq j \le M-1$ that
$$u^{(j)}=\left(\prod_{h=j}^{M-1}\frac{1}{z}\boldsymbol{X}^{(h)}\right)u^{(M)}.$$ 
Moreover, since $u \in \S_{\C}^{N-1}$, there must exist some $j \in [M]$ such that $\|u^{(j)}\|_{2}\geq \frac{1}{\sqrt{M}}.$ Hence, we have 
\begin{align*}
\frac{1}{\sqrt{M}} 
& \leq
\|u^{(j)}\|_{2}\\
&\leq \left \|\prod_{h=j}^{M-1}\frac{1}{z}\boldsymbol{X}^{(h)}\right\|_{\op}\|u^{(M)}\|_{2}\\
 & = O_{C_4,M}\left(n^{\frac{1}{2} + \frac{\epsilon_0}{32M}}\right)^{M-j}|z|^{-M+j} \|u^{(M)}\|_{2}\\
 & = O_{C_4,M}\left(n^{\frac{\epsilon_0}{16}}\right)\|u^{(M)}\|_{2},
\end{align*}
so that
$$\|u^{(M)}\|_{2} = \Omega_{C_4,M}\left(n^{-\epsilon_0/16}\right).$$

From here on, we proceed by case analysis, depending on the value of $\|\Tilde{\boldsymbol{I}}u^{(M)}\|_{2}$.\\

\textbf{Case I: }$\|\Tilde{\boldsymbol{I}}u^{(M)}\|_{2} \geq n^{-\epsilon_0/2}$.  We have  $$\|\Tilde{\boldsymbol{X}}^{(M)}u^{(1)}\|_{2} \leq \|\Tilde{\boldsymbol{X}}^{(M)}\|_{\op}\|u^{(1)}\|_{2} = O_{C_4,M}\left(n^{\frac{1}{2} + \frac{\epsilon_0}{32M}}\right)\|u^{(1)}\|_{2},$$ so that $$\frac{|z| n^{-\epsilon_0/2}}{O_{C_4,M}\left(n^{\frac{1}{2} + \frac{\epsilon_0}{32M}}\right)}\leq \frac{\|z\Tilde{\boldsymbol{I}}u^{(M)}\|_{2}}{O_{C_4,M}\left(n^{\frac{1}{2} + \frac{\epsilon_0}{32M}}\right)} \leq  \|u^{(1)}\|_{2},$$
from which we see that 
$$\|u^{(1)}\|_{2} = \Omega_{C_4,M}\left(n^{-\frac{\epsilon_0}{2} - \frac{\epsilon_0}{16M}}\right).$$

\textbf{Case II: } $\|\Tilde{\boldsymbol{I}}u^{(M)}\|_{2} < n^{-\epsilon_{0}/2}$. Since $\| u^{(M)}\|_{2} = \Omega_{C_4,M}\left(n^{-\epsilon_0/16}\right)$, it follows that $|u^{(M)}_{n}| = \Omega(\|u^{(M)}\|_{2})$. 

Let 
$$ \boldsymbol{N}:= \left(\frac{\boldsymbol{X}^{(1)}}{z}\right)\left(\frac{\boldsymbol{X}^{(2)}}{z}\right)\dots\left(\frac{\boldsymbol{X}^{(M-1)}}{z}\right).$$

Then,
$$\| \boldsymbol{N}(\Tilde{\boldsymbol{I}}u^{(M)})\|_{2} = O_{C_4,M}\left(n^{\frac{1}{2} + \frac{\epsilon_0}{32M}}\right)^{M-1}|z|^{-M+1} n^{-\epsilon_0/2} = O_{C_4,M}(n^{-7\epsilon_0/16}).$$
On the other hand, since we have restricted ourselves to the event $A$, we have
$$\|\boldsymbol{N} (u^{(M)}_{n}e_{n})\|_{2} = \Omega_{C_4, c_2, M}\left(n^{-\epsilon_0/8}\right).$$

Hence, by the triangle inequality, 
$$\| u^{(1)} \|_{2} = \left\|\left(\frac{\boldsymbol{X}^{(1)}}{z}\right)\left(\frac{\boldsymbol{X}^{(2)}}{z}\right)\dots\left(\frac{\boldsymbol{X}^{(M-1)}}{z}\right)( u^{(M)})\right\|_{2} = \Omega_{C_4,c_2,M}\left(n^{-\epsilon_0/8}\right).$$
\vspace{.3cm}

We are now ready to prove the second part of the proposition. Throughout, we restrict ourselves to the event $\mathcal{B}$, which is the intersection of the event $\mathcal{A} = \mathcal{A}_1 \cap \mathcal{A}_2$ (from \cref{eqn:A}) and the event appearing in conclusion of the first part of the proposition. 

Let $u = (u^{(1)},\dots, u^{(M)})$ be as before, and let $v = u/\|u^{(1)}\|_{2}$. Note that restricted to the event $\mathcal{B}$, $\|u^{(1)}\|_{2} \neq 0$, so that this is well defined. Moreover, by definition, we have $\|v^{(1)}\|_{2} = 1$ and since we have restricted ourselves to the event $\mathcal{B}$, we also have $\|v^{(M)}\|_{2} = \|u^{(M)}\|/\|u^{(1)}\| = \Omega_{C_4, c_2, M} (n^{-\epsilon_0})$ and $\|v^{(M)}\|_{2} = O_{C_4, c_2, M}(n^{\epsilon_0})$.

Let $a,b$ be as in the statement of the proposition. We will now show using a net argument (similar to the one in \cite{kopel2018random}) that, restricted to the event $\mathcal{B}$, $v^{(1)} \notin \Comp_{n}(a,b)$ with sufficiently high probability. For this, let $\Omega_{a,2b}$ denote a $2b$-net of $\Comp_{n}(a,b)$; by \cref{thm:net}, we can choose $\Omega_{a,2b}$ to be such that
$$|\Omega_{a,2b}| \leq \left(\frac{C_{\ref{thm:net}}}{ab}\right)^{2an}.$$
Also, let $\mathcal{N}$ be a $b$-net of $\{z \in \C:|z| \leq n^{\epsilon_0}\}$; a simple volume argument shows that $\mathcal{N}$ can be chosen so that $$ l:=|\mathcal{N}|= O(n^{2\epsilon_0}).$$ 
Let $\omega_{1},\dots,\omega_{l}$ be an enumeration of the elements in $\mathcal{N}$. 

To each $y \in \Omega_{a,2b}$, we associate vectors $x_{y,1},\dots,x_{y,l} \in \C^{n}$ satisfying  
\begin{align}
\label{eqn:system-2}
\Tilde{I}x_{y,k}&=\frac{1}{z}\Tilde{\boldsymbol{X}}^{(M)}y \\
(x_{y,k})_{n} &= \omega_{k} \nonumber.
\end{align}
Note that the (random) vectors $x_{y,1},\dots,x_{y,l}$ depend only on $\Tilde{\boldsymbol{X}}^{(M)}$. In particular, these vectors are independent of $\boldsymbol{X}^{(1)},\dots,\boldsymbol{X}^{(M-1)}$. Let
$$\Omega'_{a,2b} = \{(y,x_{y,k}) : y\in \Omega_{a,2b}, k \in [l]\}.$$ 

We now define two auxiliary events. First, let $\mathcal{E}_{1}$ denote the event (depending on $\Tilde{\boldsymbol{X}}^{(M)}$) that for all $y \in \Omega_{a,2b}$ and $k\in [l]$, $$\|x_{y,k}\|_{2} = \Omega_{c_2, C_4, M} (n^{-t_0}).$$
By using \cref{eqn:system-2}, the proof of \cref{thm:anti}, and the union bound (which is possible after taking $\epsilon_0$ to be sufficiently small depending on $c_2, C_4, M$), it follows that $\mathcal{E}_1$ holds with probability at least $1 - O_{c_2, C_4, M}(\exp(-cn))$, where $c$ is a positive constant depending on $c_2, C_4, M$. 

Second, let $\mathcal{E}_2$ denote the event that 
$$\min_{(y,x) \in \Omega'_{a,b}}\|z^{M-1}y - \boldsymbol{X}^{(1)}\dots \boldsymbol{X}^{(M-1)}x\|_{2} = \Omega_{c_2, C_4, M}\left((c_{\ref{thm:anti}}\sqrt{n})^{M-1}n^{-t_0}\right).$$
Then, a similar argument as the one used to control the event $\mathcal{A}_2$ earlier in the proof shows that 
$$\mathbb{P}\left(\mathcal{E}_2 | \mathcal{E}_1\right) \geq 1 - O_{c_2,C_4,M}(\exp(-cn)),$$
for some positive constant $c > 0$ depending only on $c_2, C_4, M$. The upshot of the discussion so far is that after losing an additive error term of at most $O_{C_4, c_2, M}(n^{-\epsilon_0/32M})$ in our final probability bound, we may restrict ourselves to the event $\mathcal{B}\cap \mathcal{E}_{2}$.

Finally, suppose that $v^{(1)} \in \Comp(a,b)$. By definition, there exists $y' \in \Omega_{a,2b}$ such that $\|v^{(1)} - y'\|_{2} \leq 2b$. Moreover, by \cref{eqn:system,eqn:system-2}, we have that for all $k \in l$,
\begin{align*}
    \|\tilde{I}v^{(M)} - \tilde{I}x_{y',k}\|_{2} 
    &= \frac{1}{|z|}\|\tilde{\boldsymbol{X}}^{(M)}(y' - v^{(1)}) \|_{2}\\
    & = O_{C_4,M}(n^{2t_0}b).
\end{align*}
Furthermore, since $\|v^{(M)}\|_{2} = O_{C_4, c_2, M}(n^{\epsilon_0})$, this upper bound is also true for the absolute value of the last coordinate of $v^{(M)}$. Therefore, using the estimate in the previous display equation along with the definition of $\mathcal{N}$, it follows that there exists some $k \in [l]$ for which
$$\|v^{(M)} - x_{y',k}\|_{2} = O_{C_4, M}(n^{2t_0} b).$$
But then
\begin{align*}
    \|z^{M-1}y' - \boldsymbol{X}^{(1)}\dots \boldsymbol{X}^{(M-1)}x'_{y,k}\|_{2}
    &\leq \|z^{M-1}(y' - v^{(1)})\|_{2} + \|z^{M-1}v^{(1)} - \boldsymbol{X}^{(1)}\dots \boldsymbol{X}^{(M-1)}v^{(M)}\|_{2}\\
    &\quad + \|\boldsymbol{X}^{(1)}\dots \boldsymbol{X}^{(M-1)}(v^{(M)} - x'_{y,k}) \|_{2}\\
    &\leq |z|^{M-1}\cdot 2b + O_{C_4, M}((\sqrt{n})^{M-1}n^{\epsilon_0/32}n^{2t_0}b)\\
    &= O_{C_4,M}((\sqrt{n})^{M-1}n^{Mt_0 + 2t_0}b)\\
    &= O_{C_4, M}((\sqrt{n})^{M-1}n^{-2t_0})
\end{align*}
since $b \leq n^{-8Mt_0}$, which cannot happen on the event $\mathcal{E}_2$.
\end{proof}

We can now quickly prove \cref{lemma:incomp} and \cref{thm:main}.
 
 \begin{proof}[Proof of \cref{lemma:incomp}]
 Let $a,b$ be as in the statement of the lemma and let $\varepsilon' = c_{\ref{lemma:incomp}}b^{4}n^{-\epsilon_0}$, where $c_{\ref{lemma:incomp}}$ is a constant depending on $c_2, C_4, M$ to be determined later. We estimate $\P(\dist_{Mn} \leq \varepsilon')$. Let $u = (u^{(1)},\dots, u^{(M)}) \in \C^{Mn}$ denote any unit vector orthogonal to the first $Mn-1$ rows of $\boldsymbol{Y}(z)$, and note that
 $$\dist_{Mn} \geq |\langle X^{(M)}_{n}, u^{(1)}\rangle - z u^{(M)}_{n}|,$$
 so that it suffices to bound the probability that the latter quantity is smaller than $\varepsilon'$. By \cref{thm:key}, we know that except with probability at most $O_{C_4, M}(n^{-\epsilon_0/32M})$, $\|u^{(1)}\|_{2} = \Omega_{C_4, c_2, M}(n^{-\epsilon_0})$ and  $u^{(1)}/\|u^{(1)}\|_{2}  \in \Incomp_{n}(a,b)$. When this is satisfied, it follows from the proof of Lemma 2.10 in \cite{livshyts2019smallest} (see also Lemma 3.6 in \cite{rudelson2009smallest}) and from Proposition 2.7 in \cite{jain2020note} that
 \begin{align*}
 \P\left(|\langle X_{n}^{(M)}, u^{(1)}/\|u^{(1)}\|_{2}\rangle - z u_{n}^{(M)}/\|u^{(1)}\|_{2}| \leq b^{4}\right) 
 &= O_{C_4, c_2, M}(b^{-2}\cdot b^{4} + \exp(-\Omega_{C_4, c_2, M}(b^{8}n^{2}/\log^{2}n)))\\
 &= O_{C_4, c_2, M}(b^{2}),
 \end{align*}
 provided that $\epsilon_0 \leq 1/1000$,
 so that
 $$\P\left(|\langle X_{n}^{(M)}, u^{(1)}\rangle - z u_{n}^{(M)}| \leq \Omega_{C_4, c_2, M}(n^{-\epsilon_0}b^{4})\right) = O_{C_4, c_2,M}(b^2).$$
 Finally, \cref{lemma:distance} gives the desired conclusion. 
 \end{proof}
 
 \begin{proof}[Proof of \cref{thm:main}]
 Let $A_{\ref{lemma:comp}}$ and $A_{\ref{lemma:incomp}}$ be the \emph{a priori} upper bounds imposed on $\epsilon_0$ in the statements of \cref{lemma:comp} and \cref{lemma:incomp} respectively, and let $A_{\ref{thm:main}} = \min\{A_{\ref{lemma:comp}}, A_{\ref{lemma:incomp}}\}$. Fix $A \in (0,A_{\ref{thm:main}})$. We apply \cref{lemma:comp} and \cref{lemma:incomp} with $\epsilon_0 = A/10$, $a = 1/\log{n}$ and $b = n^{-\epsilon_0}$ to get that
 $$\P\left(\sigma_{1}(\boldsymbol{Y}(z)) \leq c n^{-6\epsilon_0}n^{-1/2}\right) = O_{c_2, C_4, M}(n^{-\epsilon_0/32M}), $$
 where $c$ is a constant depending only on $c_2, C_4, M$. Since for all $n$ sufficiently large, $cn^{-6\epsilon_0} > n^{-A}$ and $n^{-\epsilon_0/32M} < n^{-A/1000M}$, we have the desired conclusion. 
 \end{proof}

\bibliographystyle{amsalpha}
\bibliography{biblio.bib}

\end{document}